\documentclass[reqno]{amsart}

\usepackage{graphicx}
\usepackage{amscd}
\usepackage{amsmath, amssymb}

\newtheorem{theorem}{Theorem}

\newtheorem{example}[theorem]{Example}

\newtheorem{lemma}[theorem]{Lemma}

\newtheorem{proposition}[theorem]{Proposition}

\newtheorem{question}[theorem]{Question}

\begin{document}

\title[ON THE FAITHFULNESS OF A FAMILY OF REPRESENTATIONS OF $SM_n$]{ON THE FAITHFULNESS OF A FAMILY OF REPRESENTATIONS OF THE SINGULAR BRAID MONOID $SM_n$} 

\author{Mohamad N. Nasser}

\address{Mohamad N. Nasser\\
         Department of Mathematics and Computer Science\\
         Beirut Arab University\\
         P.O. Box 11-5020, Beirut, Lebanon}
         
\email{m.nasser@bau.edu.lb}

\maketitle

\begin{abstract}
For $n\geq 2$, let $G_n$ be a group and let $\rho: B_n\rightarrow G_n$ be a representation of the braid group $B_n$. For a field $\mathbb{K}$ and $a,b,c\in \mathbb{K}$, Bardakov, Chbili, and Kozlovskaya extend the representation $\rho$ to a family of representations $\Phi_{a,b,c}:SM_n \rightarrow \mathbb{K}[G_n]$ of the singular braid monoid $SM_n$, where $\mathbb{K}[G_n]$ is the group algebra of $G_n$ over $\mathbb{K}$. In this paper, we study the faithfulness of the family of representations $\Phi_{a,b,c}$ in some cases. First, we find necessary and sufficient conditions of the families $\Phi_{a,0,0}, \Phi_{0,b,0}$ and $\Phi_{0,0,c}$ for all $n\geq 2$ to be unfaithful, where $a,b,c \in \mathbb{K}^*$. Second, we consider the case $n=2$ and we find the nature of $\ker(\Phi_{a,b,c})$ if $\Phi_{a,b,c}$ is unfaithful. Moreover, we show that there exist some families $\Phi_{a,b,c}$ that have trivial kernel in the case $n=2$. Also, we find the shape of the possible elements in $\ker(\Phi_{a,b,c})$ for all $n\geq 3$ when the kernel of ${\Phi_{a,b,c}|}_{SM_2}$ is nontrivial.
\end{abstract}

\medskip

\renewcommand{\thefootnote}{}

\vskip 0.1in

\section{Introduction} 

The braid group on $n$ strings, $B_n$, is the group with generators $\sigma_1,\sigma_2, \ldots,\sigma_{n-1}$ and a presentation as follows:
\begin{align*}
&\sigma_i\sigma_{i+1}\sigma_i = \sigma_{i+1}\sigma_i\sigma_{i+1} ,\hspace{0.65cm} i=1,2,\ldots,n-2,\\
&\sigma_i\sigma_j = \sigma_j\sigma_i , \hspace{2.25cm} |i-j|\geq 2.
\end{align*}

\vspace{0.2cm}

The monoid of singular braid, $SM_n$, is generated by the standard generators $\sigma_1^{\pm 1},\sigma_2^{\pm 1}, \ldots,\sigma_{n-1}^{\pm 1}$ of $B_n$ in addition to the singular generators $\tau_1,\tau_2, \ldots, \tau_{n-1}$. The Figures of both generators are shown in Figure 1 and Figure 2 below. $SM_n$ is introduced by J. Baez in \cite{1} and J. Birman in \cite{6}. The references \cite{8}, \cite{9}, and \cite{10} mentioned here for more information on $SM_n$.

\vspace{0.2cm}

The question of faithfulness of representations of groups and monoids has been always of a lot of significance to people working on representation theory. The most famous linear representations of $B_n$ are Burau representation \cite{7} and Lawrence-Krammer-Bigelow representation \cite{11}. Burau representation has been proved to be faithful for $n\leq 3$ \cite{5} and unfaithful for $n\geq 5$ \cite{3}; whereas the case $n=4$ remains open. Lawrence-Krammer-Bigelow representation of $B_n$ has been proved to be faithful for all $n$ \cite{4}; which shows that the group $B_n$ is linear. This leads to a natural question regarding the linearity of the singular braid monoid $SM_n$. In \cite{9}, Dasbach and Gemein constructed a linear representation of $SM_3$, which is an extension of the Burau representation of $B_3$. Moreover, they investigated extensions of the Artin representation $B_n \rightarrow
Aut(F_n)$ and the Burau representation $B_n \rightarrow GL_n(\mathbb{Z}[t^{\pm1}])$ to $SM_n$ and found connections between these representations.

\vspace{0.2cm}

In \cite{2}, Bardakov, Chbili, and Kozlovskaya introduced a family of representations of the singular braid monoid $SM_n$. More exactly, they proved that if $\rho:B_n\rightarrow G_n$ is a representation of the braid group $B_n$ to a group $G_n$ and $\mathbb{K}$ is a field, then the map $\Phi_{a,b,c}:SM_n\rightarrow \mathbb{K}[G_n]$ defined by:
\begin{align*}
&\Phi_{a,b,c}(\sigma_i^{\pm 1})=\rho(\sigma_i^{\pm 1}) \hspace{0.2cm} \text{and} \hspace{0.2cm} \Phi_{a,b,c}(\tau_i)=a\rho(\sigma_i)+b\rho(\sigma_i^{-1})+ce,\hspace{0.2cm} i=1,2,\ldots,n-1,
\end{align*}
where $a,b,c \in \mathbb{K}$ and $e$ is the neutral element of $G_n$, is a representation of $SM_n$, which is an extension of $\rho$.

\vspace{0.2cm}

One of these family of representations is the Birman representation in case $\Phi=id, a=1,b=-1$ and $c=0$. For many years, there was a conjecture that Birman representation is faithful, which is shown to be faithful by Paris in \cite{12}. 

\vspace{0.2cm}

Bardakov, Chbili, and Kozlovskaya asked the following question: 
\begin{center}
For what values of $a,b,c \in \mathbb{C}$ the representation $\Phi_{a,b,c}$ is unfaithful? \cite{2}
\end{center}
The main goal of this paper is to study the faithfulness of some family of representations $\Phi_{a,b,c}$.

\vspace{0.2cm}

In section 2 of our work, we introduce main definitions and generalities of the braid group $B_n$, the singular braid monoid $SM_n$ and the geometrical interpretation of their generators.

\vspace{0.2cm}

In section 3, we study first the faithfulness of the representations: $\Phi_{0,0,0}$, $\Phi_{a,0,0}$, $\Phi_{0,b,0}$ and $\Phi_{0,0,c}$, where $a,b,c \in \mathbb{K}^*$. First, we prove that $\Phi_{0,0,0}$ is unfaithful (Proposition \ref{propp}). Second, we find necessary and sufficient conditions of the families $\Phi_{a,0,0}$, $\Phi_{0,b,0}$ and $\Phi_{0,0,c}$ for all $n\geq 2$ to be unfaithful (Theorem \ref{Thm4}). Then, we consider the case $n=2$ and find the nature of $\ker(\Phi_{a,b,c})$ if $\Phi_{a,b,c}$ is unfaithful. More precisely, we prove that, for $n=2$, if $\Phi_{a,b,c}$ is unfaithful and has a nontrivial kernel, then $\ker(\Phi_{a,b,c})=<\tau_1^p\sigma_1^q>$ for some $p \in \mathbb{N}^*$ and $q\in \mathbb{Z}$, where $p$ is minimal positive integer (Theorem \ref{Thm5}). Moreover, we show that the kernel of some families $\Phi_{a,b,c}$ in this case is trivial (Theorem \ref{Thm6}). Then, in Theorem \ref{Thm10}, we find the shape of the possible elements in $\ker(\Phi_{a,b,c})$ for all $n\geq 2$ in the case ${\Phi_{a,b,c}|}_{SM_2}$ has nontrivial kernel. Indeed, we prove that, If ${\Phi_{a,b,c}|}_{SM_2}$ has nontrivial kernel, then the possible elements in $\ker(\Phi_{a,b,c})$ are of the form: $w=\tau_1^{r_1}v^{m_1}u_1\tau_1^{r_2}v^{m_2}u_2\ldots \tau_1^{r_k}v^{m_k}u_k$, where $v=\tau_1^p\sigma_1^q \in \ker({\Phi_{a,b,c}}_{|_{SM_2}})$ with $p$ minimal, $u_i \in B_n$, $m_i \geq 0$ and $0\leq r_i< p$, for all $1\leq i \leq k$.

\medskip

\section{Generalities} 

Recall that the braid group, $B_n$, has generators $\sigma_1,\sigma_2,\ldots,\sigma_{n-1}$ that satisfy the following relations:
\begin{align*}
&\sigma_i\sigma_{i+1}\sigma_i = \sigma_{i+1}\sigma_i\sigma_{i+1} ,\hspace{0.5cm} i=1,2,\ldots,n-2,\\
&\sigma_i\sigma_j = \sigma_j\sigma_i , \hspace{2.1cm} |i-j|>2.
\end{align*}

The singular braid monoid, introduced in \cite{1} and \cite{6}, is generated by the generators $\sigma_1^{\pm 1},\sigma_2^{\pm 1}, \ldots,\sigma_{n-1}^{\pm 1}$ of $B_n$ in addition to the singular generators $\tau_1,\tau_2, \ldots, \tau_{n-1}$. The generators $\sigma_i, \sigma_i^{-1}$ and $\tau_i$ of $SM_n$ satisfy the following relations:

\begin{align*}
(1) \hspace{1cm} \sigma_i\sigma_{i+1}\sigma_i = \sigma_{i+1}\sigma_i\sigma_{i+1} ,\hspace{.3cm} &\text{for} \hspace{.3cm} i=1,2,\ldots ,n-2,\vspace{0.1cm}\\ 
(2) \hspace{2.65cm} \sigma_i\sigma_j = \sigma_j\sigma_i , \hspace{.3cm} &\text{for} \hspace{.3cm} |i-j|\geq 2, \vspace{0.1cm}\\
(3) \hspace{2.85cm} \tau_i\tau_j = \tau_j\tau_i , \hspace{.3cm} &\text{for} \hspace{.3cm} |i-j|\geq 2, \vspace{0.1cm}\\
(4) \hspace{2.75cm} \tau_i\sigma_j=\sigma_j\tau_i, \hspace{.3cm} &\text{for} \hspace{.3cm} |i-j|\geq 2,  \vspace{0.1cm}\\
(5) \hspace{2.8cm} \tau_i\sigma_i=\sigma_i\tau_i, \hspace{.3cm} &\text{for} \hspace{.3cm} i=1,2,\ldots ,n-1,  \vspace{0.1cm}\\
(6) \hspace{1.1cm} \sigma_i\sigma_{i+1}\tau_i=\tau_{i+1}\sigma_i\sigma_{i+1}, \hspace{.3cm} &\text{for} \hspace{.3cm} i=1,2,\ldots ,n-2,  \vspace{0.1cm}\\
(7) \hspace{1.1cm} \sigma_{i+1}\sigma_{i}\tau_{i+1}=\tau_{i}\sigma_{i+1}\sigma_{i}, \hspace{.3cm} &\text{for} \hspace{.3cm} i=1,2,\ldots ,n-2.\\
\end{align*}

The geometric interpretation of the generators $\sigma_i, \sigma_i^{-1}$ and $\tau_i$ of $SM_n$ are presented in the following figures.

\begin{figure}[h]
\centering
\includegraphics[scale=0.8]{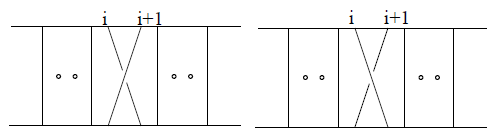}
\caption{The braid generators $\sigma_i$ and $\sigma_i^{-1}$}
\end{figure}

\begin{figure}[h]
\centering
\includegraphics[scale=0.8]{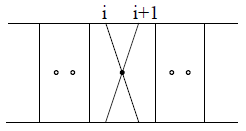}
\caption{The singular generators $\tau_i$}
\end{figure}

It is well known that the braid group, $B_n$, is generated by the two elements $\sigma_1$ and $x=\sigma_1\sigma_2\ldots \sigma_{n-1}$ for all $n\geq 2$. Also, using the relations (6) and (7) above, we can see that for all $2\leq i \leq n-1$, $\tau_i$ can be written as a product of words consist of $\tau_1$ and $\sigma_j^{\pm1}$ for some $1\leq j \leq n-1$. So, the singular braid monoid, $SM_n$, is generated by the elements $\sigma_1^{\pm 1}, x^{\pm 1}$ and $\tau_1$ for all $n\geq 2$. That is, $SM_n=<\sigma_1^{\pm 1}, x^{\pm 1},\tau_1>$ for all $n\geq 2$. In what follows, we deal with this set of generators for $SM_n$.

\medskip

\section{The Faithfulness of a Family of Representations $\Phi_{a,b,c}$ of $SM_n$} 

Consider $n\geq 2$ and let $\rho: B_n \rightarrow G_n$ be a representation of the braid group $B_n$ to a group $G_n$. Bardakov, Chbili, and Kozlovskaya extend this representation to a family of representations of the singular braid monoid $SM_n$ to a group algebra $\mathbb{K}[G_n]$ over a field $\mathbb{K}$. The following proposition is given by them in \cite{2}.

\vspace{0.2cm}

\begin{proposition} \cite{2}
Let $\rho: B_n \rightarrow G_n$ be a representation of the braid group $B_n$ to a group $G_n$ and let $\mathbb{K}$ be a field with $a,b,c \in \mathbb{K}$. Then, the map $\Phi_{a,b,c}:SM_n\rightarrow \mathbb{K}[G_n]$ which acts on the generators of $SM_n$ by the rules
\begin{align*}
&\Phi_{a,b,c}(\sigma_i^{\pm 1})=\rho(\sigma_i^{\pm 1}) \hspace{0.2cm} \text{and} \hspace{0.2cm} \Phi_{a,b,c}(\tau_i)=a\rho(\sigma_i)+b\rho(\sigma_i^{-1})+ce,\hspace{0.2cm} i=1,2,\ldots,n-1,
\end{align*}
defines a representation of $SM_n$ to $\mathbb{K}[G_n]$. Here $e$ is a neutral element of $G_n$.
\end{proposition}

\vspace{0.2cm}

Bardakov, Chbili, and Kozlovskaya asked the following question: For what values of $a,b,c \in \mathbb{C}$ the representation $\Phi_{a,b,c}$ is unfaithful? \cite{2}. In what follows, we answer this question in some cases.

\vspace{0.2cm}

First of all, we consider the case $\rho$ is unfaithful.

\vspace{0.2cm}

\begin{proposition}
Let $G_n$ be a group and let $\rho: B_n \rightarrow G_n$ be a representation. Let $\mathbb{K}$ be a field with $a,b,c \in \mathbb{K}$ and let $\Phi_{a,b,c}: SM_n\rightarrow \mathbb{K}[G_n]$ be a representation of $SM_n$  defined by:
\begin{align*}
&\Phi_{a,b,c}(\sigma_i^{\pm 1})=\rho(\sigma_i^{\pm 1}) \hspace{0.2cm} \text{and} \hspace{0.2cm} \Phi_{a,b,c}(\tau_i)=a\rho(\sigma_i)+b\rho(\sigma_i^{-1})+ce,\hspace{0.2cm} i=1,2,\ldots,n-1.
\end{align*}
If $\rho$ is unfaithful, then $\Phi_{a,b,c}$ is unfaithful for all $a,b,c \in \mathbb{K}$. 
\end{proposition}

\begin{proof}
The proof is trivial, since $\Phi_{a,b,c}$ is an extension of $\rho$.
\end{proof}

\vspace{0.2cm}

In what follows, we consider the case $\rho$ is faithful.

\vspace{0.2cm}

\begin{proposition} \label{propp}
Let $G_n$ be a group and let $\rho: B_n \rightarrow G_n$ be a faithful representation. Let $\mathbb{K}$ be a field with $a,b,c \in \mathbb{K}$ and let $\Phi_{a,b,c}: SM_n\rightarrow \mathbb{K}[G_n]$ be a representation of $SM_n$  defined by:
\begin{align*}
&\Phi_{a,b,c}(\sigma_i^{\pm 1})=\rho(\sigma_i^{\pm 1}) \hspace{0.2cm} \text{and} \hspace{0.2cm} \Phi_{a,b,c}(\tau_i)=a\rho(\sigma_i)+b\rho(\sigma_i^{-1})+ce,\hspace{0.2cm} i=1,2,\ldots,n-1.
\end{align*}
Then, $\Phi_{0,0,0}$ is unfaithful. 
\end{proposition}
\begin{proof}
We have $\Phi_{0,0,0}(\tau_1\sigma_1)=\Phi_{0,0,0}(\tau_1)\Phi_{0,0,0}(\sigma_1)=0=\Phi_{0,0,0}(\tau_1)$ with $\tau_1\sigma_1\neq \tau_1$. Thus, $\Phi_{0,0,0}$ is unfaithful.
\end{proof}

\vspace{0.2cm}

In the next Theorem, we study the faithfulness of the families $\Phi_{a,0,0}$, $\Phi_{0,b,0}$, and $\Phi_{0,0,c}$ for any $a,b,c \in \mathbb{K}^*$.

\vspace{0.2cm}

\begin{theorem} \label{Thm4}
Let $G_n$ be a group and let $\rho: B_n \rightarrow G_n$ be a faithful representation. Let $\mathbb{K}$ be a field with $a,b,c \in \mathbb{K}$ and let $\Phi_{a,b,c}: SM_n\rightarrow \mathbb{K}[G_n]$ be a representation of $SM_n$  defined by:
\begin{align*}
&\Phi_{a,b,c}(\sigma_i^{\pm 1})=\rho(\sigma_i^{\pm 1}) \hspace{0.2cm} \text{and} \hspace{0.2cm} \Phi_{a,b,c}(\tau_i)=a\rho(\sigma_i)+b\rho(\sigma_i^{-1})+ce,\hspace{0.2cm} i=1,2,\ldots,n-1.
\end{align*}
The following holds true.
\begin{itemize}
\item[(a1)] If there exists $r\in \mathbb{Z}^*$ such that $a^r=1$, then $\Phi_{a,0,0}$ is unfaithful. \\
\item[(a2)] If $a^r\neq 1$ for all $r\in \mathbb{Z}^*$, then $\Phi_{a,0,0}$ is unfaithful if and only if there exists $v\in B_n$ such that $\rho(v)=a^{-s}e$ for some $s\in \mathbb{Z}^*$.\\
\item[(b1)] If there exists $r\in \mathbb{Z}^*$ such that $b^r=1$, then $\Phi_{0,b,0}$ is unfaithful.\\ 
\item[(b2)] If $b^r\neq 1$ for all $r\in \mathbb{Z}^*$, then $\Phi_{0,b,0}$ is unfaithful if and only if there exists $v\in B_n$ such that $\rho(v)=b^{-s}e$ for some $s\in \mathbb{Z}^*$.\\
\item[(c1)] If there exists $r\in \mathbb{Z}^*$ such that $c^r=1$, then $\Phi_{0,0,c}$ is unfaithful.\\ 
\item[(c2)] If $c^r\neq 1$ for all $r\in \mathbb{Z}^*$, then $\Phi_{0,0,c}$ is unfaithful if and only if there exists $v\in B_n$ such that $\rho(v)=c^{-s}e$ for some $s\in \mathbb{Z}^*$. 
\end{itemize}
\end{theorem}

\begin{proof}
We consider each case separately.
\begin{itemize}
\item[(a1)] Suppose that there exists $r \in \mathbb{Z}^*$ such that $a^r=1$. For all $1\leq i \leq n-1$,  we have 
$$\Phi_{a,0,0}(\tau_i)=a\rho(\sigma_i),$$ and so $$\Phi_{a,0,0}(\tau_i^r)=a^r\rho(\sigma_i^r)=\rho(\sigma_i^r)=\Phi_{a,0,0}(\sigma_i^r),$$ with $\tau_i^r\neq \sigma_i^r$ because they have different geometrical shapes. Therefore, $\Phi_{a,0,0}$ is unfaithful.\\
\item[(a2)] Suppose that $a^r\neq 1$ for all $r\in \mathbb{Z}^*$. For the necessary condition, assume that $\Phi_{a,0,0}$ is unfaithful, then there exists $w_1, w_2 \in SM_n$ such that $w_1 \neq w_2$ and
$\Phi_{a,0,0}(w_1)=\Phi_{a,0,0}(w_2)$. Notice that at least one of the $w_i's$ does not belong to $B_n$ as $\rho$ is faithful. Without loss of generality, we may suppose that $w_1 \notin B_n$. Then, $w_1$ contains some $\tau_i's$ in its terms. But $\Phi_{a,0,0}(\tau_i)=a\rho(\sigma_i)$ for all $1\leq i \leq n-1$, which means that we can write $\Phi_{a,0,0}(w_1)=a^{s_1}\Phi_{a,0,0}(v_1)$, where $v_1 \in B_n$ and $s_1 \in \mathbb{N}^*$ is the number of times $\tau_i's$ occurs in $w$. Since $v_1 \in B_n$, it follows that $\Phi_{a,0,0}(w_1)=a^{s_1}\Phi_{a,0,0}(v_1)=a^{s_1}\rho(v_1)$. Similarly, we can see that $\Phi_{a,0,0}(w_2)=a^{s_2}\rho(v_2)$ where $v_2 \in B_n$ and $s_2 \in \mathbb{N}$ is the number of times $\tau_i's$ occurs in $w$ ($s_2$ may be $0$). But $\Phi_{a,0,0}(w_1)=\Phi_{a,0,0}(w_2)$, so $a^{s_1}\rho(v_1)=a^{s_2}\rho(v_2)$. Hence, $\rho(v_1v_2^{-1})=a^{-s_1}a^{s_2}e=a^{-s_1+s_2}e$. Remark that we may easily prove that $s_1\neq s_2$, since otherwise we get $v_1=v_2$ and then $w_1=w_2$, which is a contradiction. Therefore, there exists $v=v_1v_2^{-1} \in B_n$ such that $\rho(v)=a^{-s}e$, where $s=-s_1+s_2 \in \mathbb{Z}^*$, as required. 

Now, for the sufficient condition, suppose that there exist $v\in B_n$ and $s\in \mathbb{Z}^*$ such that $\rho(v)=a^{-s}e$. Then, for all $1 \leq i \leq n-1$, $\Phi_{a,0,0}(\tau_i^sv)=\Phi_{a,0,0}(\tau_i^s)\Phi_{a,0,0}(v)=a^s\rho(\sigma_i^s)\rho(v)=a^s\rho(\sigma_i^s)a^{-s}e=\rho(\sigma_i^s)=\Phi_{a,0,0}(\sigma_i^s)$. Notice that we can easily see that 
$\tau_i^sv \neq \sigma_i^s$, since otherwise we get $v^{-1}\sigma_i^{s}=\tau_i^s$ with $v^{-1}\sigma_i^{s}\in B_n$, which is impossible. Hence, $\rho$ is unfaithful as required.
\end{itemize}
The proof of (b1), (b2), (c1), and (c2) is typical.
\end{proof}

Now, we consider the case $n=2$ and study the kernel of $\Phi_{a,b,c}$ in this case. First of all, we find the nature of $\ker(\Phi_{a,b,c})$ in the case $\Phi_{a,b,c}:SM_2\rightarrow \mathbb{K}[G_2]$ is unfaithful.

\vspace{0.2cm}

\begin{theorem} \label{Thm5}
Let $G_2$ be a group and let $\rho: B_2 \rightarrow G_2$ be a faithful representation. Let $\mathbb{K}$ be a field with $a,b,c \in \mathbb{K}$ and let $\Phi_{a,b,c}: SM_2\rightarrow \mathbb{K}[G_2]$ be a representation of $SM_2$  defined by:
\begin{align*}
&\Phi_{a,b,c}(\sigma_1^{\pm 1})=\rho(\sigma_1^{\pm 1}) \hspace{0.2cm} \text{and} \hspace{0.2cm} \Phi_{a,b,c}(\tau_1)=a\rho(\sigma_1)+b\rho(\sigma_1^{-1})+ce,\hspace{0.2cm} i=1,2,\ldots,n-1.
\end{align*}
If $\Phi_{a,b,c}$ is unfaithful and $ker(\Phi_{a,b,c})$ is not trivial , then $\ker(\Phi_{a,b,c})=<\tau_1^p\sigma_1^q>$ for some $p \in \mathbb{N}^*$ and $q\in \mathbb{Z}$, where $p$ is minimal positive integer.
\end{theorem}

\begin{proof}
Since $SM_2=<\sigma_1^{\pm 1},\tau_1>$ and $\sigma_1$ commute with $\tau_1$, it follows that any word in $SM_2$ can be written as: $\tau_1^p\sigma_1^q$ for some $p\in \mathbb{N}$ and $q\in \mathbb{Z}$. Now, as $\ker(\Phi_{a,b,c})$ is not trivial, we pick $v=\tau_1^p\sigma_1^q \in \ker(\Phi_{a,b,c})$ to be a nontrivial element such that $p \in \mathbb{N}^*$ and $q\in \mathbb{Z}$, where $p$ is minimal positive integer. The positive integer $p$ should be nonzero as $\rho$ is faithful. Now, we require to prove that $\ker(\Phi_{a,b,c})=<v>$. Let $u \in \ker(\Phi_{a,b,c})$ be a nontrivial element, then $u=\tau_1^r\sigma_1^s$, where $r\in \mathbb{N}^*$ and $s\in \mathbb{Z}$. Since $p$ is minimal, it follows that $p \leq r$. We consider the following two cases:
\begin{itemize}
\item[(a)] If $p=r$, then $u=\tau_1^r\sigma_1^s=\tau_1^p\sigma_1^s=\tau_1^p\sigma_1^q\sigma_1^{s-q}=v\sigma_1^{s-q}$. So, $\Phi_{a,b,c}(u)=\Phi_{a,b,c}(v)\Phi_{a,b,c}(\sigma_1^{s-q})$. But $u,v \in \ker(\Phi_{a,b,c})$ gives that $\Phi_{a,b,c}(u)=\Phi_{a,b,c}(v)=e$, and so $\Phi_{a,b,c}(\sigma_1^{s-q})=e$. Note that $\Phi_{a,b,c}(\sigma_1^{s-q})=\rho(\sigma_1^{s-q})$ and $\rho$ is faithful, so $\sigma_1^{s-q}=e$. But we know that $\sigma_1^m \neq e$ for all $m \in \mathbb{Z}^*$ as $B_2$ is torsion free, so $s-q=0$ and so $s=q$. Thus, $u=\tau_1^r\sigma_1^s=\tau_1^p\sigma_1^q=v$.
\item[(b)] Suppose that $p<r$. Then, there exists $m \in \mathbb{N}^*$ such that $r=mp+(r-mp)$, where $0\leq r-mp< p$. Now, $u=\tau_1^r\sigma_1^s=\tau_1^{r-mp}\tau_1^{mp}\sigma_1^{s-mq}\sigma_1^{mq}=\tau_1^{mp}\sigma_1^{mq}\tau_1^{r-mp}\sigma_1^{s-mq}=v^m\tau_1^{r-mp}\sigma_1^{s-mq}$. But $\Phi_{a,b,c}(u)=\Phi_{a,b,c}(v)=e$ implies that $\Phi_{a,b,c}(\tau_1^{r-mp}\sigma_1^{s-mq})=e$. Hence, $\tau_1^{r-mp}\sigma_1^{s-mq} \in \ker(\Phi_{a,b,c})$ where $0\leq r-mp< p$. But $p$ is a minimal positive integer, so $r-mp=0$ and so $r=mp$. Repeat the same work in (a) we get $s=mq$. Thus, $u=\tau_1^r\sigma_1^s=\tau_1^{mp}\sigma_1^{mq}=(\tau_1^{p}\sigma_1^{q})^m=v^m$.
\end{itemize}
Hence, for all $u \in \ker(\Phi_{a,b,c})$, there exists $m \in \mathbb{N}$ such that $u=v^m$. Therefore, $\ker(\Phi_{a,b,c})=<v>$, as required.
\end{proof}

\vspace*{0.2cm}

Remark that, the condition "$\ker(\Phi_{a,b,c})$ is not trivial" is a must in the previous theorem. Since in Monoid Representation Theory in general, we may have unfaithful monoid representations with trivial kernel. This lead to the following question.

\vspace*{0.2cm}

\begin{question}
Can we find a representation $\rho: B_n \rightarrow G_n$ and a field $\mathbb{K}$ with $a,b,c \in \mathbb{K}$ in a way that the extension $\Phi_{a,b,c}$ of $\rho$ is unfaithful representation with trivial kernel? 
\end{question}

\vspace*{0.2cm}

Now, we study the kernel of the representation $\Phi_{a,b,c}:SM_2\rightarrow \mathbb{K}[G_2]$ in some cases.

\vspace*{0.2cm}

\begin{theorem} \label{Thm6}
Let $G_2$ be a group and let $\rho: B_2 \rightarrow G_2$ be a faithful representation. Let $\mathbb{K}$ be a field with $a,b,c \in \mathbb{K}$ and let $\Phi_{a,b,c}: SM_2\rightarrow \mathbb{K}[G_2]$ be a representation of $SM_2$  defined by:
\begin{align*}
&\Phi_{a,b,c}(\sigma_1^{\pm 1})=\rho(\sigma_1^{\pm 1}) \hspace{0.2cm} \text{and} \hspace{0.2cm} \Phi_{a,b,c}(\tau_1)=a\rho(\sigma_1)+b\rho(\sigma_1^{-1})+ce,\hspace{0.2cm} i=1,2,\ldots,n-1.
\end{align*}
\begin{itemize}
\item[(i)] Suppose that for all $s\in \mathbb{Z}^*$, $\rho(\sigma_1^s)\neq de$ for all $d \in \mathbb{K}$, then $\ker(\Phi_{a,b,c})$ is trivial.
\item[(ii)] If there exists $d \in \mathbb{K}$ such that $\rho(\sigma_1)=de$, then $\ker(\Phi_{a,b,c})$ is nontrivial if and only if there exists $p \in \mathbb{N}^*$ and $q\in \mathbb{Z}$ with $\displaystyle \sum_{i+j+k=p \atop i,j,k \in \mathbb{N}} \frac{p!}{i!j!k!}a^ib^jc^kd^{i-j+q}-1=0.$
\end{itemize}
\end{theorem}

\begin{proof}
We consider each case in the following.
\begin{itemize} 
\item[(i)] Let $u=\tau_1^s\sigma_1^r$ be a nontrivial element in $SM_2$. Suppose to get a contradiction that $u \in \ker(\Phi_{a,b,c})$. We have $\Phi_{a,b,c}(u)=\Phi_{a,b,c}(\tau_1^s)\Phi_{a,b,c}(\sigma_1^r)=(a\rho(\sigma_1)+b\rho(\sigma_1^{-1})+ce)^s\rho(\sigma_1)^r= \displaystyle \sum_{i+j+k=s \atop i,j,k \in \mathbb{N}} \frac{s!}{i!j!k!}a^ib^jc^k\rho(\sigma_1)^{i-j+r}=e$. So, $\displaystyle \sum_{i+j+k=s \atop i,j,k \in \mathbb{N}} \frac{s!}{i!j!k!}a^ib^jc^k\rho(\sigma_1)^{i-j+r}-e=0$. Now, as for all $s\in \mathbb{Z}^*$, $\rho(\sigma_1^s)\neq de$ for all $d \in \mathbb{K}$, we see that the terms of the sum  $\displaystyle \sum_{i+j+k=s \atop i,j,k \in \mathbb{N}} \frac{s!}{i!j!k!}a^ib^jc^k\rho(\sigma_1)^{i-j+r}-e$ are all different elements of $\mathbb{K}[G_n]$. Moreover, the terms $\rho(\sigma_1)^{i-j+r}$ are all different for all $i,j,k \in \mathbb{N}$ with $i+j+k=s$, and they can not be equal to the neutral element $e$. But $\mathbb{K}[G_n]$ is a group algebra over the field $\mathbb{K}$, hence, each coefficient of this sum is $0$, and so the coefficient of $e$ in the sum is zero, which is a contradiction. Thus $w \notin \ker(\Phi_{a,b,c})$, and so $\ker(\Phi_{a,b,c})$ is trivial.
\item[(ii)] Consider an element $d \in \mathbb{K}$ such that $\rho(\sigma_1)=de$. For the necessary condition, suppose  $\ker(\Phi_{a,b,c})$ is nontrivial, then, by Theorem \ref{Thm5}, there exists $p \in \mathbb{N}^*$ and $q\in \mathbb{Z}$ such that $\ker(\Phi_{a,b,c})=<\tau_1^p\sigma_1^q>$. So, $\Phi_{a,b,c}(\tau_1^p\sigma_1^q)=\displaystyle \sum_{i+j+k=p \atop i,j,k \in \mathbb{N}} \frac{p!}{i!j!k!}a^ib^jc^k\rho(\sigma_1)^{i-j+q} =\displaystyle \sum_{i+j+k=p \atop i,j,k \in \mathbb{N}} \frac{p!}{i!j!k!}a^ib^jc^kd^{i-j+q}e=e$, and so $\displaystyle \sum_{i+j+k=p \atop i,j,k \in \mathbb{N}} \frac{p!}{i!j!k!}a^ib^jc^kd^{i-j+q}-1=0$ as required. The sufficient condition can be proved similarly.
\end{itemize}
\end{proof}

\begin{example}
\noindent As an example for the above Theorem, we take the Birman representation defined in \cite{12} for $n=2$. We can see that the Birman representation satisfy the condition of Theorem \ref{Thm6} (i), since $\rho=id$, and so for all $s\in \mathbb{Z}^*$, $\rho(\sigma_1^s)=\sigma_1^s\neq de$ for all $d \in \mathbb{K}$ because $\sigma_1^s$ and $e$ have different geometrical shapes. So, we have $\Phi_{a,b,c}$ has trivial kernel, which is proved by Paris in \cite{12}.
\end{example}

\vspace*{0.2cm}

\begin{question}
Under which conditions we get $\Phi_{a,b,c}$ is faithful in Theorem \ref{Thm6} (i)?
\end{question}

\vspace*{0.2cm}

Notice that the remaining case in Theorem \ref{Thm6} is when $\rho(\sigma_1)\neq de$ for all $d\in \mathbb{K}$ and there exists $s\in \mathbb{Z}^*$ such that $\rho(\sigma_1^s)=d_se$ for some $d_s \in \mathbb{K}$. We reduce the result of this case in the following proposition.

\vspace*{0.2cm}

\begin{proposition} \label{prop1}
Let $G_2$ be a group and let $\rho: B_2 \rightarrow G_2$ be a faithful representation. Let $\mathbb{K}$ be a field with $a,b,c \in \mathbb{K}$ and let $\Phi_{a,b,c}: SM_2\rightarrow \mathbb{K}[G_2]$ be a representation of $SM_2$  defined by:
\begin{align*}
&\Phi_{a,b,c}(\sigma_1^{\pm 1})=\rho(\sigma_1^{\pm 1}) \hspace{0.2cm} \text{and} \hspace{0.2cm} \Phi_{a,b,c}(\tau_1)=a\rho(\sigma_1)+b\rho(\sigma_1^{-1})+ce,\hspace{0.2cm} i=1,2,\ldots,n-1.
\end{align*}
Suppose that $\rho(\sigma_1)\neq de$ for all $d\in \mathbb{K}$ and there exists $s\in \mathbb{Z}^*$ such that $\rho(\sigma_1^s)=d_se$ for some $d_s \in \mathbb{K}$. Choose $s\in \mathbb{Z}^*$ such that $|s|$ is minimal. Let $H_2=<\rho(\sigma_1),\rho(\sigma_1)^2,\ldots, \rho(\sigma_1)^{s-1}>$ be a subgroup of $G_2$ and consider the representation $\Psi_{a,b,c}: SM_2\rightarrow \mathbb{K}[H_2]$ defined by: $\Psi_{a,b,c}(w)=\Phi_{a,b,c}(w)$ for all $w\in SM_n$. Then, $\ker(\Phi_{a,b,c})=\ker(\Psi_{a,b,c})$.
\end{proposition}
\begin{proof}
Trivial.
\end{proof}

\vspace*{0.2cm}

\begin{question}
Under which conditions we get $\ker(\Phi_{a,b,c})$ is trivial in Proposition \ref{prop1}? Also, under which conditions we get $\Phi_{a,b,c}$ is faithful in Proposition \ref{prop1}?
\end{question}

\vspace*{0.2cm}

Now, we consider the case $n\geq 2$ in order to study the shape of the possible elements in $\ker(\Phi_{a,b,c})$. First of all, we consider the following two Lemmas.

\vspace*{0.2cm}

\begin{lemma} \label{lemm8}
Let $G_n$ be a group and let $\rho: B_n \rightarrow G_n$ be a representation. Let $\mathbb{K}$ be a field with $a,b,c \in \mathbb{K}$ and let $\Phi_{a,b,c}: SM_n\rightarrow \mathbb{K}[G_n]$ be a representation of $SM_n$  defined by:
\begin{align*}
&\Phi_{a,b,c}(\sigma_i^{\pm 1})=\rho(\sigma_i^{\pm 1}) \hspace{0.2cm} \text{and} \hspace{0.2cm} \Phi_{a,b,c}(\tau_i)=a\rho(\sigma_i)+b\rho(\sigma_i^{-1})+ce,\hspace{0.2cm} i=1,2,\ldots,n-1.
\end{align*}
Then, $w \in \ker(\Phi_{a,b,c})$ if and only if $uwu^{-1}\in \ker(\Phi_{a,b,c})$ for all $u\in B_n$.
\end{lemma}
\begin{proof}
Trivial.
\end{proof}

\vspace*{0.2cm}

\begin{lemma} \label{lemm9}
Let $G_n$ be a group and let $\rho: B_n \rightarrow G_n$ be a representation. Let $\mathbb{K}$ be a field with $a,b,c \in \mathbb{K}$ and let $\Phi_{a,b,c}: SM_n\rightarrow \mathbb{K}[G_n]$ be a representation of $SM_n$  defined by:
\begin{align*}
&\Phi_{a,b,c}(\sigma_i^{\pm 1})=\rho(\sigma_i^{\pm 1}) \hspace{0.2cm} \text{and} \hspace{0.2cm} \Phi_{a,b,c}(\tau_i)=a\rho(\sigma_i)+b\rho(\sigma_i^{-1})+ce,\hspace{0.2cm} i=1,2,\ldots,n-1.
\end{align*}
Then, each word $w\in SM_n$ can be written as: $w=\tau_1^{r_1}u_1\tau_1^{r_2}u_2 \ldots \tau_1^{r_k}u_k$, where $r_i\in \mathbb{N}$ and $u_i \in B_n$ for all $1\leq i\leq k.$
\end{lemma}
\begin{proof}
Let $w \in SM_n=<\sigma_1^{\pm 1}, x^{\pm 1}, \tau_1>$ be nontrivial. If $w$ starts with $\tau_1$, then the proof is straightforward. Suppose that $w$ starts with some $u \in B_n$, then by Lemma \ref{lemm8}, $u^{-1}wu$ is an element in the kernel which starts with $\tau_1$, and so the proof is completed.
\end{proof}

\vspace*{0.2cm}

Now, we find the shape of the possible elements in $\ker(\Phi_{a,b,c})$ in the case $\ker({\Phi_{a,b,c}|}_{SM_2})$ is nontrivial.

\vspace*{0.2cm}

\begin{theorem} \label{Thm10}
Let $G_n$ be a group and let $\rho: B_n \rightarrow G_n$ be a representation. Let $\mathbb{K}$ be a field with $a,b,c \in \mathbb{K}$ and let $\Phi_{a,b,c}: SM_n\rightarrow \mathbb{K}[G_n]$ be a representation of $SM_n$  defined by:
\begin{align*}
&\Phi_{a,b,c}(\sigma_i^{\pm 1})=\rho(\sigma_i^{\pm 1}) \hspace{0.2cm} \text{and} \hspace{0.2cm} \Phi_{a,b,c}(\tau_i)=a\rho(\sigma_i)+b\rho(\sigma_i^{-1})+ce,\hspace{0.2cm} i=1,2,\ldots,n-1.
\end{align*}
If $\ker({\Phi_{a,b,c}|}_{SM_2})$ is nontrivial, then the possible elements in $\ker(\Phi_{a,b,c})$ are of the form: $w=\tau_1^{r_1}v^{m_1}u_1\tau_1^{r_2}v^{m_2}u_2\ldots \tau_1^{r_k}v^{m_k}u_k$, where $v=\tau_1^p\sigma_1^q \in \ker({\Phi_{a,b,c}}_{|_{SM_2}})$ with $p$ minimal, $u_i \in B_n$, $m_i \geq 0$ and $0\leq r_i< p$, for all $1\leq i \leq k$. Moreover, the word $w=\tau_1^{r_1}u_1\tau_1^{r_2}u_2\ldots \tau_1^{r_k}u_k$ is also in $\ker(\Phi_{a,b,c})$.
\end{theorem}

\begin{proof}
Suppose that $\ker({\Phi_{a,b,c}|}_{SM_2})$ is nontrivial. Then, by Theorem \ref{Thm5}, $\ker({\Phi_{a,b,c}|}_{SM_2})= <\tau_1^p\sigma_1^q>$ for some $p \in \mathbb{N}^*$ and $q\in \mathbb{Z}$, where $p$ is minimal. Let $w \in \ker(\Phi_{a,b,c})$, then by Lemma \ref{lemm9}, $w$ can be written as $w=\tau_1^{s_1}u_1\tau_1^{s_2}u_2\ldots \tau_1^{s_k}u_k$, where $s_i\in \mathbb{N}$ and $u_i \in B_n$ for all $1\leq i\leq k.$ If $s_1\geq p$, then there exists $m_1\in \mathbb{N}$ such that $s_1=m_1p+(s_1-m_1p)$, where $0\leq s_1-m_1p<p$, and so we see that $w=\tau_1^{s_1}u_1\tau_1^{s_2}u_2\ldots \tau_1^{s_k}u_k=\tau_1^{s_1-m_1p}\tau_1^{m_1p}\sigma_1^{m_1q}\sigma_1^{-m_1q}u_1\tau_1^{s_2}u_2\ldots \tau_1^{s_k}u_k=\tau_1^{s_1-m_1p}v^m\sigma_1^{-m_1q}u_1\tau_1^{s_2}u_2\ldots \tau_1^{s_k}u_k=\tau_1^{r_1}v^mu'_1\tau_1^{s_2}u_2\ldots \tau_1^{s_k}u_k$, where $r_1=s_1-m_1p<p$ and $u_1'=\sigma_1^{-m_1q}u_1\in B_n$. We do the same for all $s_i$, $2\leq i\leq k$, and we get the required result.
\end{proof}

\vspace*{0.2cm}

\begin{question}
To continue the result of Theorem \ref{Thm10}, we may ask many questions.
\begin{itemize}
\item[(1)] can we eliminate some elements from $\ker({\Phi_{a,b,c}})$ in the case $\ker({\Phi_{a,b,c}|}_{SM_2})$ is nontrivial?
\item[(2)] What is the shape of possible elements in $\ker(\Phi_{a,b,c})$ if $\ker({\Phi_{a,b,c}|}_{SM_2})$ is trivial?
\end{itemize}
\end{question}

\medskip

\section{Conclusion} 
In this paper, we studied the faithfulness of a family of representations ${\Phi_{a,b,c}}$ of the singular braid monoid $SM_n$. The first main important result is that we answered the question of Bardakov, Chbili, and Kozlovskaya on the faithfulness of ${\Phi_{a,b,c}}$ when two of the parameters $a,b$ and $c$ are zeros. On the other hand, the second main important result is that we found the nature of $\ker(\Phi_{a,b,c})$ if $\Phi_{a,b,c}$ is unfaithful and its kernel is nontrivial in the case $n=2$. Moreover, we made a relation between the faithfulness of ${\Phi_{a,b,c}}$ in the case $n>2$ and the case $n=2$. That is, we found the shape of the possible elements in $\ker(\Phi_{a,b,c})$ for all $n>2$ if the kernel of the restriction of $\Phi_{a,b,c}$ to $SM_2$ is nontrivial. There are a lot of issues to continue as a future work as we mentioned in many questions in Section 3.

\vspace{0.2cm}

\textbf{Conflict of Interest.} The author declare no conflict of interest.


\end{document}